\documentclass[graybox]{svmult}

\usepackage{type1cm}        
\usepackage{makeidx}         
\usepackage{graphicx}        
\usepackage{multicol}        
\usepackage[bottom]{footmisc}

\usepackage{newtxtext}       %
\usepackage{newtxmath}       
\newtheorem{thm}{Theorem}[section]

 \newtheorem{hypothesis}[thm]{Assumptions (H)}

\usepackage{xcolor}

\makeindex             

\begin{document}

\title*{Relative entropy in diffusive relaxation\\ for a class of discrete velocities BGK models}
\author{Roberta Bianchini}
\institute{Roberta Bianchini \at Sorbonne Universit\'e, LJLL, UMR CNRS-UPMC 7598, 4 place Jussieu, 75252-Paris Cedex 05
(France) \& Consiglio Nazionale delle Ricerche, IAC, via dei Taurini 19 - 00185 Rome (Italy), \email{r.bianchini@iac.cnr.it}}
%
%
\maketitle

\abstract{We provide a general framework to extend the relative entropy method to a class of diffusive relaxation systems with discrete velocities. The methodology is detailed in the toy case of the 1D Jin-Xin model under the diffusive scaling, and provides a direct proof of convergence to the limit parabolic equation in any interval of time, in the regime where the solutions are smooth. Recently, the same approach has been successfully used to show the strong convergence of a vector-BGK model to the 2D incompressible Navier-Stokes equations.}

\section{Introduction}
\label{sec:intro}
The model under investigation in this paper is the so-called \emph{Jin-Xin relaxation system}, first introduced in \cite{XinJin}. This model is likely the simplest example of \emph{hyperbolic relaxation systems}, i.e. semilinear approximations to hyperbolic systems of balance laws, see for instance \cite{Mascia} for a review on the topic. We are interested in the diffusive-scaled version of this model, which reads as follows:
\begin{equation}
\label{scaled_JinXin}
\begin{cases}
\partial_t u + \partial_x v=0, \\
\varepsilon^2 \partial_t v+ \lambda^2 \partial_x u=f(u)-v,
\end{cases}
\end{equation}
where $\lambda>0$ is a positive constant, $u, v:(t, x) \in \mathbb{R}^+\times \mathbb{R} \rightarrow \mathbb{R}$, and $f(u): \mathbb{R} \rightarrow \mathbb{R}$ is a Lipschitz function satisfying the hypotheses below.
\begin{hypothesis}
\label{Assumptions_h}
$f(u): \mathbb{R} \rightarrow \mathbb{R}$ is a Lipschitz function such that:
\begin{itemize}
\item $f(0)=0;$
\item $f'(0)=a,$ which is a constant value independent of $\varepsilon, \lambda;$
\item $f(u)=au+h(u),$ where $h(u)$ is a polynomial function of order higher than or equal to 2.
\end{itemize}
\end{hypothesis}
After a quick glance, it is clear that the equilibrium of system (\ref{scaled_JinXin}) has to satisfy in a suitable sense the limit parabolic equation
\begin{equation}\label{eq:limit-parabolic}
\partial_t \bar{u}+\partial_x f(\bar{u})=\lambda^2 \partial_{xx} \bar{u}.
\end{equation}
We will use the relative entropy method for diffusive relaxation to establish this convergence in a rigorous way.\\
\indent Despite its simplicity, the Jin-Xin system (\ref{scaled_JinXin}), and in particular its hyperbolic-scaled version, has been intensively studied as a very good test case for relaxation models, first introduced by Whitham, \cite{Whitham} and then generalized by Chen, Levermore and Tai-Ping Liu \cite{CLL}. Far from being complete, we provide here some comments on a very shortened list of previous works. For the hyperbolic Jin-Xin system approximating scalar equations, we mention \cite{ChenLiu, XinJin, NataliniCPAM} and references within \cite{Mascia}, while the case of systems was developed in \cite{Stefano}. For the diffusive scaling considered in (\ref{scaled_JinXin}), the first investigations started with the papers \cite{Kurtz, McKean} on the Carleman equation, it was futher applied by Marcati and collaborators for the analysis of hyperbolic-parabolic relaxation limits for weak solutions of hyperbolic systems of balance laws \cite{Marcati}, while a two-velocities model is studied in \cite{Lions} by Lions and Toscani. We also quote \cite{BGN}, where the diffusion limit of (\ref{BGK_scaled_JinXin}) was studied by using monotonicity properties of the solution, and \cite{LinLu}, where initial data around a traveling wave were considered. 
The interest in studying the one-dimensional Jin-Xin model in this work is due to its equivalence to the BGK (Bhatnagar–Gross–Krook) discrete kinetic approximation of conservation laws, first established by Natalini in \cite{Natalini}. In the rest of the paper, we will mostly consider the BGK version of system (\ref{scaled_JinXin}), which is written below and is equivalent to (\ref{scaled_JinXin}) in the smooth regime. We point out that this equivalence between relaxation and discrete kinetic approximations of hyperbolic balance laws only holds in one space dimension.\\
The BGK formulation of system (\ref{scaled_JinXin}) is obtained by applying the change of variables
\begin{equation}
\label{change_BGK}
u=f_1^\varepsilon+f_2^\varepsilon, \qquad v=\dfrac{\lambda}{\varepsilon}(f_1^\varepsilon-f_2^\varepsilon),
\end{equation}
to system (\ref{scaled_JinXin}), which gives
\begin{equation}
\label{BGK_scaled_JinXin}
\begin{cases}
\partial_t f_1^\varepsilon+\dfrac{\lambda}{\varepsilon}\partial_x f_1^\varepsilon=\dfrac{1}{\varepsilon^2}(M_1(u)-f_1^\varepsilon), \\
\partial_t f_2^\varepsilon-\dfrac{\lambda}{\varepsilon}\partial_x f_2^\varepsilon=\dfrac{1}{\varepsilon^2}(M_2(u)-f_2^\varepsilon), \\
\end{cases}
\end{equation}
where the Maxwellian functions are given by:
\begin{equation}\label{Maxwellians_JinXin}
M_1(u)=\dfrac{u}{2}+\dfrac{\varepsilon f(u)}{2\lambda}, \qquad M_2(u)=\dfrac{u}{2}-\dfrac{\varepsilon f(u)}{2\lambda}.
\end{equation}
For later purposes, we introduce the hyperbolic-scaled version of the Jin-Xin model, which was investigated in several previous works \cite{NataliniCPAM, Chern, XinJin} and reads as follows:
\begin{equation}
\label{hyperbolic_Jin_Xin}
\begin{cases}
\partial_t u + \partial_x v=0, \\
\varepsilon^2 (\partial_t v+ \lambda^2 \partial_x u)=f(u)-v.
\end{cases}
\end{equation}
As proved in \cite{NataliniCPAM}, the equilibrium solution solves the conservation law
\begin{equation}
\label{hyperbolic_limit_JinXin}
\partial_t \bar{u}+\partial_x f(\bar{u})=0.
\end{equation}
A key point for proving convergence of (\ref{hyperbolic_Jin_Xin}) to the equilibrium solution is the \emph{subcharacteristic condition}, which is widely discussed in \cite{Mascia} and requires that the equilibrium velocities are controlled by the discrete $\lambda$, i.e.
$$\lambda \ge |f'(u)|, \; \text{for any   } u.$$
The role of this condition is to ensure stability, since a straightforward Chapman-Enskog expansion of system (\ref{hyperbolic_Jin_Xin}) shows that 
$$(\lambda^2-(f'(u))^2) \ge 0$$
is the coefficient of the parabolic term. A very nice connection of the subcharacteristic condition with BGK approximations for hyperbolic balance laws, introduced in \cite{Natalini} and generalized in \cite{Bouchut}, was found in \cite{Natalini}. In \cite{Natalini}, the author proves that a crucial hypothesis to get convergence of the BGK system to equilibrium (\ref{hyperbolic_limit_JinXin}) is the \emph{monotonicity of the Maxwellians, which is actually equivalent to the subcharacteristic condition}. The existence of a positively invariant region for the Jacobians of the Maxwellians and the existence of a family of convex entropies $\mathcal{E}$ for the hyperbolic equilibrium (\ref{hyperbolic_limit_JinXin}) were the ingredients to prove convergence of the Jin-Xin system, established by Serre in \cite{Serre}. At the same time, Bouchut in \cite{Bouchut} extended that result to a general class of BGK models. The main idea behind this is that \emph{the positivity of Jacobians of the the Maxwellians induces a diffeomorphims beetween the space of density functions $f_i$ and the equilibrium ${u}$}. This roughly tells us that we can construct convex entropies for the relaxation system, starting from a family of convex entropies for the hyperbolic equilibrium (\ref{hyperbolic_limit_JinXin}), as proved in \cite{Bouchut}.\\
In the explicit case of system (\ref{BGK_scaled_JinXin}), the Jacobians of the Maxwellians are given by
\begin{equation}
\label{Maxwellian_derivatives}
M_1'(u)=\dfrac{1}{2}+\dfrac{\varepsilon f'(u)}{2\lambda}, \qquad M_2'(u)=\dfrac{1}{2}-\dfrac{\varepsilon f'(u)}{2\lambda},
\end{equation}
therefore there exists a fixed $\varepsilon_0>0$ such that they are strictly positive for every $\varepsilon \le \varepsilon_0$, if $|f'(u)|$ is bounded. This satisfies the assumptions of Theorem 2.1 in \cite{Bouchut}, which provides the existence of a kinetic convex entropy $\mathcal{H}(\textbf{f}^{\; \varepsilon})=\mathcal{H}_1(f_1^\varepsilon)+\mathcal{H}_2(f_2^\varepsilon)$ for the diffusive Jin-Xin system in BGK formulation (\ref{BGK_scaled_JinXin}), where $\textbf{f}^{\; \varepsilon}=(f_1^\varepsilon, f_2^\varepsilon)$ has been defined in (\ref{change_BGK}). \\Denoting by $\textbf{M}(u)=(M_1(u), \; M_2(u))$ with $M_i(u), \, i=1,2$ in (\ref{Maxwellians_JinXin}), Theorem 2.1 in \cite{Bouchut} also states that
\begin{equation}
\label{mininum_H}
H(\textbf{f}^{\;\varepsilon}) \ge H(\textbf{M}(u)),
\end{equation}
which, together with the convexity property, provides the following entropy estimate
\begin{equation}
\label{1entropy inequality}
\partial_t \mathcal{H}(\textbf{f}^{\;\varepsilon})+\dfrac{\Lambda}{\varepsilon}\partial_x \mathcal{H}(\textbf{f}^{\;\varepsilon})=\dfrac{1}{\varepsilon^2} \mathcal{H}'(\textbf{f}^{\;\varepsilon}) (\textbf{M}(u)-\textbf{f}^{\;\varepsilon}) \le \dfrac{1}{\varepsilon^2} (\mathcal{H}(\textbf{M}(u))-\mathcal{H}(\textbf{f}^{\; \varepsilon})) \le 0.
\end{equation}
This entropy inequality can be seen as a discrete version of the Boltzmann H-Theorem and tells us that
\emph{the kinetic entropy $\mathcal{H}(\textbf{f}^{\; \varepsilon})$ is dissipative}, see for instance \cite{HN} for a rigorous definition.
More precisely, properties (E1)-(E2) and Theorem 2.1 in \cite{Bouchut} assure that, for any $\eta(u) \in \mathcal{E}$, defining the projector
\begin{equation}
\label{Projector}
\mathcal{P}=\left(\begin{array}{cc}
1 & 1 \\
0 & 0
\end{array}\right), \quad \text{such that} \quad \mathcal{P}\textbf{f}^{\; \varepsilon}=f_1^\varepsilon+f_2^\varepsilon=u,
\end{equation}
according with (\ref{change_BGK}),
\begin{equation}
\label{minimu_H_precise}
\eta(u)=\min_{\mathcal{P}\textbf{f}^{\; \varepsilon} = u} \mathcal{H}(\textbf{f}^{\; \varepsilon})=\mathcal{H}(\textbf{M}(u)).
\end{equation}
In this context, the Gibbs principle for relaxation implies that
\begin{equation}
\label{orthogonality_Tzavaras}
\mathcal{H}'(\textbf{M}(u)) \perp Ker(\mathcal{P}).
\end{equation}
Since $\textbf{f}^{\; \varepsilon}-\textbf{M}(u) \in Ker(\mathcal{P})$, the convexity property of $\mathcal{H}(\textbf{f}^{\; \varepsilon})$ together with condition (\ref{orthogonality_Tzavaras}) allow us to get the following inequality:
\begin{equation}
\label{dissipative_entropy}
\mathcal{H}'(\textbf{f}^{\; \varepsilon}) (\textbf{f}^{\; \varepsilon}-\textbf{M}(u)) \le -c |\textbf{f}^{\; \varepsilon}-\textbf{M}(u)|^2, \qquad c=c(\textbf{f}^{\; \varepsilon}),
\end{equation}
namely the kinetic entropy $\mathcal{H}(\textbf{M}(u))$ is \emph{strictly dissipative}, according to the definition given for instance in \cite{HN}.
This way, the entropy inequality for system (\ref{scaled_JinXin}) reads
\begin{equation}
\label{strictly_entropy_inequality}
\partial_t \mathcal{H}(\textbf{f}^{\; \varepsilon})+\dfrac{\Lambda}{\varepsilon}\partial_x \mathcal{H}(\textbf{f}^{\; \varepsilon}) \le - \dfrac{c}{\varepsilon^2}|\textbf{f}^{\; \varepsilon}-\textbf{M}(u)|^2.
\end{equation}
Together with the ones already mentioned before, other two inspiring papers for the present work are \cite{Tzavaras, Lattanzio}. In \cite{Tzavaras}, Tzavaras proved convergence of singular hyperbolic relaxation systems, endowed with a family of convex entropies, under the hyperbolic scaling. The proof relies on the  \emph{relative entropy method} and holds in the context of smooth solutions. We recall that the more general method of modulated energy was used by Brenier in \cite{Brenier} for the Vlasov-Poisson system, and see also \cite{Laure} for an interesting application to fluid-dynamics. Moreover, in \cite{Lattanzio}, Lattanzio and Tzavaras extended this method to diffusive relaxation, in the particular case of 1D compressible gas dynamics with friction.\\
This paper can be seen as an extension of \cite{Tzavaras} to the diffusive scaling, for a class of discrete velocities relaxation models, whose hyperbolic equilibrium admits a family of convex entropies. 
Here we choose to present the method through the 1D diffusive Jin-Xin system (\ref{BGK_scaled_JinXin}), since this basic model admits a simple and explicit convex entropy, which is a relaxation of the entropy of the hyperbolic equilibrium satisfying (\ref{hyperbolic_limit_JinXin}). We point out that in the case of more relevant BGKs, whose cone of the positive Maxwellians is invariant, the kinetic convex entropy provided by Theorem 2.1 in \cite{Bouchut} is not explicit, since it comes from the application of the Inverse Function Theorem. This is a major difficulty in applying the relative entropy method in that case, since the expression of the kinetic entropy is in general not known and depends on the vanishing diffusive parameter. This obstacle has been overcome in \cite{Bianchini3}, where the relative entropy method was applied to a vector-BGK model approximating the two-dimensional Navier-Stokes equations. The guidelines of the method developed in \cite{Bianchini3} are contained in the present work on the diffusive Jin-Xin model. We stress that although the approach is presented here for a scalar equation (\ref{limit_parabolic}), it can be easily extended to systems, as done in \cite{Bianchini3}.
\\
The result of this paper obtained by the relative entropy method is stated here.
\begin{theorem}\label{thm-main}
Let $\bar{u}$ be a smooth solution to 
$$\partial_t \bar{u}+\partial_x f(\bar{u})=\lambda^2 \partial_{xx} \bar{u},$$
with initial data $u_0 \in H^{s+1}(\mathbb{R}) \; s > \frac{1}{2}+1$, and 
consider a family of smooth solutions $\textbf{f}^{\; \varepsilon}$ to the discrete velocities diffusive BGK model (\ref{BGK_scaled_JinXin}), emanating from well-prepared initial data
$$\textbf{f}^{\; \varepsilon}_0=(\mathcal{M}_1(u_0), \, \mathcal{M}_2(u_0))=(\frac{u_0}{2}+\frac{\varepsilon f(u_0)}{2\lambda}-\frac{\varepsilon \lambda}{2}\partial_x u_0, \, \frac{u_0}{2}-\frac{\varepsilon f(u_0)}{2\lambda}+\frac{\varepsilon \lambda}{2}\partial_x u_0).$$
For any interval of time $[0, T]$, the following strong convergence result holds:
\begin{align*}
\sup_{t \in [0, T]} \|(u-\bar{u})(t)\|_{s'} \le c \varepsilon^{\frac{1}{2}-\frac{s'}{2s}},
\end{align*}
with $s' \in (0, s)$.
\end{theorem}
We end this introduction with some comments on the long-time behavior of the Jin-Xin system, which still provides convergence to a parabolic equation as (\ref{limit_parabolic}), with $\lambda^2-a^2$ instead of $\lambda^2$ as the diffusion coefficient. In this regard, the paper 
\cite{Zuazua} is a complete study of the long-time behavior of the Jin-Xin system for a class of functions $f(u) = |u|^{q-1} u, \; \text{with  } q \ge 2$. In \cite{BHN}, the authors extended this investigation to dissipative hyperbolic systems with convex entropies. Finally, the convergence of the Jin-Xin model (\ref{scaled_JinXin}) to the limit parabolic equation, both for large times and for the vanishing singular parameter, is proved in \cite{Bianchini} for a general class of functions $f(u)=au+h(u)$, with $h(u)$ quadratic.\\
\section{The relative entropy}
Let $(\bar{u}, \; \bar{v})$ be a smooth solution to the limit diffusive Jin-Xin system in (\ref{scaled_JinXin}). It is well-known, see for instance \cite{BGN, Bianchini}, that $(\bar{u}, \; \bar{v})$ solves the parabolic equation
\begin{equation}
\label{limit_parabolic}
\partial_t \bar{u}+\partial_x f(\bar{u})=\lambda^2 \partial_{xx}\bar{u}, \qquad \text{and} \qquad \bar{v}=f(\bar{u})-\lambda^2 \partial_x \bar{u}.
\end{equation}
The equilibria $\bar{u}, \bar{v}$ can be translated in terms of the densities,
\begin{equation}
\label{limit_f}
\bar{u}=\bar{f}_1+\bar{f}_2, \quad \bar{v}=\dfrac{\lambda}{\varepsilon}(\bar{f}_1-\bar{f}_2),
\end{equation}
where the densities at (diffusive) equilibrium are given by suitable corrections of the Maxwellians, as discussed for instance in \cite{GS} in the case of the Boltzmann equation and in \cite{Bianchini} for our model,
\begin{equation}
\label{limit_Maxwellians}
\bar{f}_1=M_1(\bar{u})-\dfrac{\varepsilon \lambda \partial_x \bar{u}}{2}=: \mathcal{M}_1(\bar{u}), \qquad  \bar{f}_2=M_2(\bar{u})+\dfrac{\varepsilon \lambda \partial_x \bar{u}}{2}=: \mathcal{M}_2(\bar{u}),
\end{equation}
and the expressions of $M_1(\bar{u}), M_2(\bar{u})$ are provided in (\ref{Maxwellians_JinXin}). We recall from the previous section that under the hypothesis of positivity of the Jacobians of the Maxwellians in (\ref{Maxwellian_derivatives}), system (\ref{BGK_scaled_JinXin}) admits a dissipative convex kinetic entropy.
Under such a framework, taking a smooth solution $\textbf{f}^{\; \varepsilon}=(f_1^\varepsilon, \, f_2^\varepsilon)$ to system (\ref{BGK_scaled_JinXin}) and a smooth solution $\bar{u}$ to (\ref{limit_parabolic}), a relative entropy is defined as follows,
\begin{equation}
\label{relative_entropy_def}
\tilde{\mathcal{H}}(\textbf{f}^{\; \varepsilon}|\bar{\textbf{f}})=\mathcal{H}(\textbf{f}^{\; \varepsilon})-\mathcal{H}(\mathcal{M}(\bar{u}))-\mathcal{H}'(\mathcal{M}(\bar{u})) \cdot (\textbf{f}^{\; \varepsilon}-\mathcal{M}(\bar{u})).
\end{equation}
We also define the relative entropy flux,
\begin{equation}
\label{relative_entropy_flux}
\begin{aligned}
\tilde{\mathcal{Q}}(\textbf{f}^{\; \varepsilon}| \bar{\textbf{f}})&=\dfrac{\lambda}{\varepsilon}(\mathcal{H}_1(f_1^\varepsilon)-\mathcal{H}_2(f_2^\varepsilon))-\dfrac{\lambda}{\varepsilon}(\mathcal{H}_1(\mathcal{M}_1(\bar{u}))-\mathcal{H}_2(\mathcal{M}_2(\bar{u})))\\
&-\dfrac{\lambda}{\varepsilon}\mathcal{H}'_1(\mathcal{M}_1(\bar{u})) \cdot (f_1^\varepsilon-\mathcal{M}_1(\bar{u}))+\dfrac{\lambda}{\varepsilon}\mathcal{H}'_2(\mathcal{M}_2(\bar{u})) \cdot (f_2^\varepsilon-\mathcal{M}_2(\bar{u})).
\end{aligned}
\end{equation}
The aim of this section is to provide the existence of a kinetic convex dissipative entropy for system (\ref{BGK_scaled_JinXin}).
\subsection{Overview on the linear case}
In the linear case, it is particularly easy to find an expression of a dissipative kinetic convex entropy for the Jin-Xin system and all the computations are explicit. Consider the diffusive Jin-Xin model with linear source term
\begin{equation}
\label{Jin-Xin_linear}
\begin{cases}
& \partial_t u + \dfrac{\partial_x(\varepsilon^2 v)}{\varepsilon^2}=0, \\
& \partial_t(\varepsilon^2 v)+\lambda^2 \partial_x u = au-\dfrac{\varepsilon^2 v}{\varepsilon^2},
\end{cases}
\end{equation}
approximating the linear convection-diffusion equation
\begin{equation}\label{eq:limit-conv-diff}
\partial_t \bar{u}+a \partial_x \bar{u}=\lambda^2 \partial_{xx}\bar{u}.
\end{equation}
Introducing $\textbf{u}=(u, \; \varepsilon^2 v)$, system (\ref{Jin-Xin_linear}) rewrites as
\begin{equation}
\label{JinXin_compact_linear}
\partial_t \textbf{u} + A \partial_x \textbf{u} = -B \textbf{u},
\end{equation}
where 
\begin{equation}
\label{matrices_Jin_Xin_compact_form_linear}
A=\left(\begin{array}{cc}
0 & \dfrac{1}{\varepsilon^2} \\
\lambda^2 & 0 \\
\end{array}\right), \qquad 
-B=\left(\begin{array}{cc}
0 & 0 \\
a & -\dfrac{1}{\varepsilon^2}
\end{array}\right).
\end{equation}
One can symmetrize system (\ref{JinXin_compact_linear}), by using a positive definite symmetric matrix
\begin{equation}
\label{symmetrizer_JinXin}
\Sigma=\left(\begin{array}{cc}
1 & a \varepsilon^2 \\
a \varepsilon^2 & \lambda^2 \varepsilon^2 \\
\end{array}\right), \quad \text{with    } A\Sigma = (A\Sigma)^T,
\end{equation}
and such that the dissipation of $B$ in (\ref{matrices_Jin_Xin_compact_form_linear}) is enhanced, see \cite{Bianchini, Bianchini1} for a detailed discussion,
\begin{equation*}
-B_1=-B\Sigma=\left(\begin{array}{cc}
0 & 0 \\
0 & a^2 \varepsilon^2 - \lambda^2
\end{array}\right).
\end{equation*}
Now, as usual, see for instance \cite{Benzoni},
\begin{equation}
\label{Sigma_inverse}
\Sigma^{-1}=\dfrac{1}{\lambda^2-a^2\varepsilon^2}\left(\begin{array}{cc}
\lambda^2 & -a \\
-a & \frac{1}{\varepsilon^2}
\end{array}\right)
\end{equation}
is a classical left symmetrizer for system (\ref{JinXin_compact_linear}), and so it can be viewed as the Hessian matrix of an entropy function for system (\ref{JinXin_compact_linear}), whose expression is given by
\begin{equation}
\label{entropy_Jin-Xin_linear}
E(\textbf{u})=\eta(u, \varepsilon^2 v)=\dfrac{\lambda^2}{2}u^2+\dfrac{\varepsilon^2}{2} v^2-au\varepsilon^2v.
\end{equation}
In terms of the kinetic variables, 
\begin{equation}
\label{entropy_Jin-Xin_f}
\begin{aligned}
E(u, \varepsilon^2 v)&=h(f_1^\varepsilon+f_2^\varepsilon, \, \lambda (f_1^\varepsilon-f_2^\varepsilon)/\varepsilon)\\
&=(f_1^\varepsilon)^2(\lambda^2-a\lambda \varepsilon) + (f_2^\varepsilon)^2(\lambda^2+a\lambda \varepsilon)=:\mathcal{H}_1(f_1^\varepsilon)+\mathcal{H}_2(f_2^\varepsilon)=\mathcal{H}(\textbf{f}^{\; \varepsilon}),
\end{aligned}
\end{equation}
i.e. $\mathcal{H}(\textbf{f}^{\; \varepsilon})$ is an explicit convex and dissipative kinetic entropy for the linear BGK model (\ref{BGK_scaled_JinXin}), with Maxwellian functions
\begin{equation}
\label{Maxwellians_linear_Jin-Xin}
M_1(u)=\frac{u}{2}+\frac{a\varepsilon u}{2\lambda}, \quad M_2(u)=\frac{u}{2}-\frac{a\varepsilon u}{2\lambda}.
\end{equation}
We recall that the existence of such an entropy $\mathcal{H}(\textbf{f}^{\; \varepsilon})$ is provided by the theory developed in \cite{Bouchut}.
According to Theorem 2.1 in \cite{Bouchut}, which tells us that the kinetic entropy is a relaxation of the entropy of the conservation law and that they coincide at equilibrium,
\begin{equation}
\begin{cases}
\mathcal{H}(M(\textbf{u}))=\mathcal{H}_1(M_1(u))+\mathcal{H}_2(M_2(u))=\eta(u)=\dfrac{1}{2}(\lambda^2-a^2\varepsilon^2)u^2, \\
\partial_{f_i^\varepsilon}\mathcal{H}_i(M_i(u))=\eta'(u)=(\lambda^2-a^2\varepsilon^2)u, \quad i=1,2.
\end{cases}
\end{equation}
However, a perturbed version of the original Maxwellians, $\mathcal{M}_1(u), \mathcal{M}_2(u)$ in (\ref{limit_Maxwellians}), naturally arises as equilibria of the diffusive-scaled system. In the linear case,
\begin{equation}
\label{perturbed_Max_linear_case}
\mathcal{M}_1(u)=\dfrac{u}{2}+\dfrac{a\varepsilon u}{2}-\dfrac{\varepsilon \lambda\partial_x u}{2}, \qquad \mathcal{M}_1(u)=\dfrac{u}{2}-\dfrac{a\varepsilon u}{2}+\dfrac{\varepsilon \lambda\partial_x u}{2},
\end{equation}
which implies that
\begin{equation}
\label{entropy_perturbed_Max}
\mathcal{H}_1'(\mathcal{M}_1(u))=\eta'(u)-\varepsilon \lambda^3 \partial_x u + O(\varepsilon^2), \qquad \mathcal{H}'_2(\mathcal{M}_2(u))=\eta'(u)+\varepsilon \lambda^3 \partial_x u + O(\varepsilon^2),
\end{equation}
where
$$\eta''(u)=\lambda^2+O(\varepsilon).$$
\subsection{The nonlinear case}\label{sec:nonlinear}
We go back to the nonlinear one-dimensional Jin-Xin model (\ref{scaled_JinXin}) and its BGK formulation (\ref{BGK_scaled_JinXin}).
In the nonlinear case, the existence of a kinetic, convex and dissipative entropy for the BGK system (\ref{BGK_scaled_JinXin}), in the vicinity of the local equilibrium $\textbf{f}^{\; \varepsilon}=\textbf{M}(u)$ is due to Theorem 2.1 in \cite{Bouchut}. However, Theorem 2.1 in \cite{Bouchut} only provides the existence of such an entropy, while its explicit expression is not known in the general case, being indeed a consequence of the Inverse Function Theorem, see \cite{Bianchini3} as an example with non-explicit kinetic entropies. In the particular case of the 1D Jin-Xin system, an explicit entropy can be shown. Nevertheless, in the following we will not rely on the expression of the kinetic entropy, in order to present the method in full generality.
The results below will be applied later on.
\begin{proposition}\label{lemma-positive-Max}
Consider the BGK system (\ref{BGK_scaled_JinXin}) under Assumptions (H).
Let $\eta(\bar{u})$ be a quadratic entropy for the hyperbolic limit of the Jin-Xin model (\ref{hyperbolic_limit_JinXin}). Then there exists a convex kinetic entropy $\mathcal{H}(\textbf{f}^{\; \varepsilon})$ for system (\ref{BGK_scaled_JinXin}), such that (\ref{minimu_H_precise})-(\ref{dissipative_entropy}) hold, as long as $u=f_1^\varepsilon+f_2^\varepsilon$ is bounded in $[0, T]$, for any fixed constant $T$.
\end{proposition}
\begin{proof}
We apply Theorem 2.1 in \cite{Bouchut}, which yields the result provided that
the derivatives of the Maxwellians $M_i(u), \; i=1,2$ in (\ref{Maxwellians_JinXin}) are positive, 
$$M_1'(u)=\dfrac{1}{2}+\dfrac{\varepsilon f'(u)}{2\lambda}>0 , \qquad M_2'(u)=\dfrac{1}{2}-\dfrac{\varepsilon f'(u)}{2\lambda}>0,$$
which hold as long as $|u(t)|_\infty=|f_1^\varepsilon(t)+f_2^\varepsilon(t)|_\infty$ is bounded. 
\end{proof}
\begin{lemma}
\label{lemma_2_derivative_entropy}
Let $\eta(\bar{u})$ be a quadratic entropy for the conservation law (\ref{hyperbolic_limit_JinXin}). Under Assumptions (H), there exists an interval of time $[0, T^*]$, where $T^*$ depends on $\|u_0\|_{s+1}, \; s>\frac{1}{2}+1$, such that system (\ref{BGK_scaled_JinXin}) admits a kinetic convex dissipative entropy  $$\mathcal{H}(\textbf{f}{\;^\varepsilon})=\mathcal{H}_1(f_1^\varepsilon)+\mathcal{H}_2(f_2^\varepsilon),$$ with
$$\mathcal{H}(\textbf{M}(u))=\eta(u), \qquad \mathcal{H}_1(M_1(u))=\mathcal{H}_2(M_2(u))=\eta'(u).$$
Moreover, 
\begin{equation}
\label{2_derivative_entropy}
|\mathcal{H}_i''(M_i(u))|_\infty = 2 |\eta''(u)|_\infty+O(\varepsilon), \quad i=1,2.
\end{equation}
\end{lemma}
\begin{proof}
The idea is to prove that the assumptions of Proposition \ref{lemma-positive-Max} are satisfied. Therefore we need to show that there exists a fixed interval of time $[0, T^*]$, with $T^*$ independent of $\varepsilon$, such that $|u(t)|_\infty$ is bounded for $t \le T^*$. 
For the diffusive Jin-Xin model, this is actually already proved (for global times) in \cite{Bianchini}, using a detailed description of the Green function of the linearized system. However, we perform here a classical local in time proof in order to be self-contained. \\
Defining $\textbf{w}$ such that $\textbf{u}=\Sigma \textbf{w}$, the compact system is symmetric and reads
\begin{align*}
\partial_t \textbf{w}+\tilde{A} \partial_x \textbf{w}=-\tilde{B} \textbf{w}
+\begin{pmatrix}
0\\
h((\Sigma\textbf{w})_1),
\end{pmatrix}
\end{align*}
where
\begin{align*}
\tilde{A}=A\Sigma=\tilde{A}^T=
\begin{pmatrix}
a & \lambda^2 \\
\lambda^2 & a \varepsilon^2 \lambda^2
\end{pmatrix},  \quad
\tilde{B}=B\Sigma=
\begin{pmatrix}
0 & 0 \\
0 & \lambda^2-a^2 \varepsilon^2
\end{pmatrix}.
\end{align*}
Moreover, simple computations show that $\Sigma$ is positive definite and, more precisely,
\begin{equation}\label{ineq:sigma-w}
\frac{1}{2}\|w_1\|_0^2 + \varepsilon^2 (\lambda^2 - 2a^2 \varepsilon^2) \|w_2\|_0^2 \le (\Sigma \textbf{w}, \textbf{w})_0 \le (1+a\varepsilon^2) \|w_1\|_0^2 + \varepsilon^2 (a+\lambda^2)\|w_2\|_0^2.
\end{equation}
Then one gets the following energy estimates on the symmetric compact system:
\begin{equation}\label{ineq:energy-estimate-s}
\begin{aligned} 
\frac{1}{2}\|w_1(t)\|_s^2 & + \varepsilon^2 (\lambda^2 - 2a^2 \varepsilon^2) \|w_2(t)\|_s^2 + (\lambda^2-a^2\varepsilon^2) \int_0^t \|w_2(\tau)\|_s^2 \, d\tau \\
& \le \|u_0\|_s^2 + \varepsilon^2 c(\|u_0\|_{s+1})\\
& + c(\|w_1+a\varepsilon^2 w_2\|_{L^\infty_{t,x}}^2) \int_0^t \|w_1(\tau)  + a \varepsilon^2 w_2(\tau) \|_s^2 \, d \tau;
\end{aligned}
\end{equation}
\begin{align*}
\frac{1}{2}\|\partial_t w_1(t)\|_{s-1}^2 & + \varepsilon^2 (\lambda^2 - 2a^2 \varepsilon^2) \|\partial_t w_2(t)\|_{s-1}^2 + (\lambda^2-a^2\varepsilon^2) \int_0^t \|\partial_\tau w_2(\tau)\|_{s-1}^2 \, d\tau \\
& \le \|\partial_t u|_{t=0}\|_{s-1}^2\\
& + c(\|w_1+a\varepsilon^2 w_2\|_{L^\infty_{t,x}}^2) \int_0^t \|\partial_\tau w_1(\tau)  + a \varepsilon^2 \partial_\tau w_2(\tau) \|_{s-1}^2 \, d\tau \\
& \le \|\partial_x f(u_0) - \lambda^2 \partial_{xx} u_0\|_{s-1}^2\\
& + c(\|w_1+a\varepsilon^2 w_2\|_{L^\infty_{t,x}}^2) \int_0^t \|\partial_\tau w_1(\tau)  + a \varepsilon^2 \partial_\tau w_2(\tau) \|_{s-1}^2 \, d\tau \\
\end{align*}
\begin{equation}\label{ineq:energy-estimate-s-1-time-derivative}
\begin{aligned}
& \le c\|u_0\|_{s+1}^2 \\
& + c(\|w_1+a\varepsilon^2 w_2\|_{L^\infty_{t,x}}^2) \int_0^t \|\partial_\tau w_1(\tau)  + a \varepsilon^2 \partial_\tau w_2(\tau) \|_{s-1}^2 \, d\tau.
\end{aligned} 
\end{equation}
This provides
\begin{align*}
\frac{1}{2}\|w_1(t)\|_s^2 + \varepsilon^2 (\lambda^2 - 2a^2 \varepsilon^2) \|w_2(t)\|_s^2  \le (\|u_0\|_s^2 + \varepsilon^2 c(\|u_0\|_{s+1})) e^{c(\|w_1+a\varepsilon^2 w_2\|_{L^\infty_{t,x}}^2)t},\\
\frac{1}{2}\|\partial_t w_1(t)\|_{s-1}^2 + \varepsilon^2 (\lambda^2 - 2a^2 \varepsilon^2) \|\partial_t w_2(t)\|_{s-1}^2 \le c\|u_0\|_{s+1}^2 e^{c(\|w_1+a\varepsilon^2 w_2\|_{L^\infty_{t,x}}^2)t},
\end{align*}
as an application of the Gronwall inequality.
Now set $\|u_0\|_{s+1}=:M_0$ and, for a fixed constant $M>M_0$, define
\begin{equation}\label{def-time}
T^*:=\sup_{ t \in [0, T_\varepsilon)} \Bigg\{\frac{1}{2}\|w_1(t)\|_\infty + \varepsilon^2 (\lambda^2 - 2a^2 \varepsilon^2) \|w_2(t)\|_\infty \le c_S M \Bigg\},
\end{equation}
being $T_\varepsilon$ the maximum existence time of $\textbf{f}^{\; \varepsilon}$ and $c_S$ the Sobolev embedding constant.
The Sobolev Embedding Theorem provides the following estimates for $t \in [0, T^*],$
\begin{equation}\label{estimate-M}
\frac{1}{2}\|w_1(t)\|_\infty^2 + \varepsilon^2 (\lambda^2 - 2a^2 \varepsilon^2) \|w_2(t)\|_\infty^2 \le (M_0^2+\varepsilon^2 c(M_0)) e^{c(M^2)t} \le c_S^2 M^2, 
\end{equation}
\begin{equation}\label{estimate-space-derivative-M}
\frac{1}{2}\|\partial_x w_1(t)\|_\infty^2 + \varepsilon^2 (\lambda^2 - 2a^2 \varepsilon^2) \|\partial_x w_2(t)\|_\infty^2 \le c M_0^2 e^{c(M^2)t} \le c_S^2 M^2, 
\end{equation}
\begin{equation}\label{estimate-time-derivative-M}
\frac{1}{2}\|\partial_t w_1(t)\|_\infty^2 + \varepsilon^2 (\lambda^2 - 2a^2 \varepsilon^2) \|\partial_t w_2(t)\|_\infty^2 \le c M_0^2 e^{c(M^2)t} \le c_S^2 M^2,
\end{equation}
where $T^* \le \dfrac{1}{c(M^2)}\log  \Bigg(\dfrac{c_S^2 M^2}{cM_0^2}\Bigg).$\\
Now, using the changes of variables \eqref{change_BGK} and $\textbf{u}=(u, \varepsilon^2 v)=\Sigma \textbf{w}$, one gets
\begin{align*}
f_1^\varepsilon=\dfrac{w_1}{2 \lambda}(\lambda+a\varepsilon) + \frac{\varepsilon w_2}{2}(\lambda+a \varepsilon), \\
f_1^\varepsilon=\dfrac{w_1}{2 \lambda}(\lambda-a\varepsilon) - \frac{\varepsilon w_2}{2}(\lambda - a \varepsilon),
\end{align*}
and so, thanks to (\ref{estimate-M}), 
$$\sup_{t \in [0, T^*]} \max\{ \|f_i^\varepsilon(t)\|_\infty, \; i=1,2\} \le c(M),$$
which allows us to apply Proposition \ref{lemma-positive-Max} to prove the first of part of this lemma. It remains to show the last expansion.
Notice from (\ref{estimate-space-derivative-M})-(\ref{estimate-time-derivative-M}) that $|\varepsilon \partial_t f_i^\varepsilon + \lambda \partial_x f_i^\varepsilon)|$ is bounded in $[0, T^*]$, then
\begin{align*}
& f_1^\varepsilon=M_1(u)+\varepsilon(\varepsilon \partial_t f_1^\varepsilon + \lambda \partial_x f_1^\varepsilon)=M_1(u)+O(\varepsilon)=\frac{u}{2}+\frac{\varepsilon f(u)}{2 \lambda} + O(\varepsilon),\\
& f_2^\varepsilon=M_2(u)+\varepsilon(\varepsilon \partial_t f_2^\varepsilon - \lambda \partial_x f_2^\varepsilon)=M_2(u)+O(\varepsilon)=\frac{u}{2}-\frac{\varepsilon f(u)}{2 \lambda} + O(\varepsilon),
\end{align*}
in the sense of the $L^\infty_x$ norm. From Proposition \ref{lemma-positive-Max},
$$\partial_{f_i^\varepsilon}H_i(f_i^\varepsilon)|_{f_i^\varepsilon=M_i(u)}=\partial_{f_i^\varepsilon}H_i(M_i(u))=\partial_{f_i^\varepsilon}H_i(\frac{u}{2}\pm \varepsilon \frac{f(u)}{2 \lambda})=g(\frac{u}{2}\pm \varepsilon \frac{f(u))}{2 \lambda})=\eta'(u).$$
Recalling that $u=f_1^\varepsilon+f_2^\varepsilon$, 
$$\partial^2_{f_i^\varepsilon}H_i(f_i^\varepsilon) = \frac{1}{2}g'(\frac{u}{2}+O(\varepsilon))=\eta''(u),$$
i.e.
$$\partial^2_{f_i^\varepsilon}H_i(f_i^\varepsilon)=2\eta''(u)+O(\varepsilon).$$
\end{proof}
\begin{lemma}\label{lemma-perturbed-maxwellians}
Let $$\mathcal{H}(\textbf{f}^{\; \varepsilon})=\mathcal{H}_1(f_1^\varepsilon)+\mathcal{H}_2(f_2^\varepsilon)$$ be a kinetic convex dissipative entropy associated with system (\ref{BGK_scaled_JinXin}). Let $\bar{u}$ be a smooth solution to system (\ref{limit_parabolic}) and $\eta(\bar{u})$ a convex entropy for equation (\ref{hyperbolic_limit_JinXin}). Then, for $i=1,2$,
\begin{equation*}
\mathcal{H}_i(\mathcal{M}_i(\bar{u}))=\mathcal{H}_i (M_i(\bar{u})\mp \dfrac{\varepsilon \lambda}{2}\partial_x \bar{u})=\mathcal{H}_i(M_i(\bar{u}))\mp \eta'(\bar{u}) \dfrac{\varepsilon \lambda}{2}\partial_x \bar{u}+O(\varepsilon^2),
\end{equation*}
\begin{equation*}
\mathcal{H}_i'(\mathcal{M}_i(\bar{u}))=\eta'(\bar{u}) \mp \varepsilon \lambda \eta''(\bar{u}) \partial_x \bar{u}+O(\varepsilon^2),
\end{equation*}
in the sense of the $L^\infty$ norm.
\end{lemma}
\begin{proof}
It comes directly from a Taylor expansion and the use of $\mathcal{H}_i'(M_i(u))=\eta'(u)$ for any $u$.
\end{proof}
\section{Relative entropy estimate}
The global in time convergence is achieved by means of the relative entropy estimate, which is stated and proved below.
\begin{theorem}\label{thm:relative-entropy-estimate}
Consider the discrete velocities BGK model (\ref{BGK_scaled_JinXin}) under Assumptions (H), endowed with a kinetic convex dissipative entropy $\mathcal{H}(\textbf{f}^{\; \varepsilon})$ for $t \le T^*$ fixed. Let $\bar{u}$ be a smooth solution to
$$\partial_t \bar{u}+\partial_x f(\bar{u})=\lambda^2 \partial_{xx}\bar{u},$$
with initial data $u_0 \in H^{s+1}(\mathbb{R}), \; s>\frac{1}{2}+1$, and let $\textbf{f}^{\; \varepsilon}(t)=(f_1^\varepsilon, \, f_2^\varepsilon), \; t \le T^*$ be a sequence of solutions to the discrete velocities BGK model \eqref{BGK_scaled_JinXin}, emanating from smooth well-prepared initial data 
$$\textbf{f}^{\; \varepsilon}_0=(\mathcal{M}_1(u_0), \, \mathcal{M}_2(u_0)),$$
and such that 
$$|\textbf{f}^{\; \varepsilon}(t)|_\infty \, |\partial_t \textbf{f}^{\; \varepsilon}(t)|_\infty \le M, \; t \le T^*, \;  \text{for any constant   } M>0 \text{  independent of  } \varepsilon.$$
Then, the following stability estimate holds:
\begin{equation}\label{eq:relative-entropy-estimate}
\sup_{t \in [0, T^*]} \|u-\bar{u}\|_0 \le c \sqrt{\varepsilon} e^{c(|\bar{u}|_\infty, \, |\partial_x \bar{u}|_\infty)t}.
\end{equation}
\end{theorem}
\begin{proof}
In the following, we drop the apex $\varepsilon$ for simplicity.
\begin{align*}
\partial_t \tilde{\mathcal{H}}(\textbf{f}^{\; \varepsilon}| \overline{\textbf{f}})&+\partial_x \tilde{\mathcal{Q}}(\textbf{f}^{\; \varepsilon}|\bar{\textbf{f}})\\
&=\partial_t (\mathcal{H}_1(f_1)+\mathcal{H}_2(f_2))+\dfrac{\lambda}{\varepsilon}\partial_x (\mathcal{H}_1(f_1)-\mathcal{H}_2(f_2))\\
&-\partial_t (\mathcal{H}_1(\overline{\mathcal{M}}_1)+\mathcal{H}_2(\overline{\mathcal{M}}_2))-\dfrac{\lambda}{\varepsilon} \partial_x (\mathcal{H}_1(\overline{\mathcal{M}}_1)-\mathcal{H}_2(\overline{\mathcal{M}}_2))\\
&-\partial_t (\mathcal{H}_1'(\overline{\mathcal{M}}_1) (f_1-\overline{\mathcal{M}}_1))-\partial_t (\mathcal{H}_2'(\overline{\mathcal{M}}_2) (f_2-\overline{\mathcal{M}}_2))\\
&-\dfrac{\lambda}{\varepsilon}\partial_x(\mathcal{H}_1'(\overline{\mathcal{M}}_1)(f_1-\overline{\mathcal{M}}_1))+\dfrac{\lambda}{\varepsilon}\partial_x(\mathcal{H}_2'(\overline{\mathcal{M}}_1)(f_2-\overline{\mathcal{M}}_2))\\
&=:I_1+I_2+I_3+I_4.
\end{align*}
We start with $I_1$, whose estimate is based on Lemma \ref{lemma_2_derivative_entropy}.
\begin{equation}
\label{I1}
\begin{aligned}
I_1&=\partial_t (\mathcal{H}_1(f_1)+\mathcal{H}_2(f_2))+\dfrac{\lambda}{\varepsilon}\partial_x (\mathcal{H}_1(f_1)-\mathcal{H}_2(f_2)) \\
&=\dfrac{\mathcal{H}_1'(f_1)}{\varepsilon^2} (M_1-f_1) + \dfrac{\mathcal{H}_2'(f_2)}{\varepsilon^2} (M_2-f_2)\\
&=\dfrac{\mathcal{H}_1'(M_1)+\mathcal{H}_1''(M_1) (f_1-M_1)}{\varepsilon^2} (M_1-f_1)\\
& + \dfrac{\mathcal{H}_2'(M_2)+\mathcal{H}''(M_2) (f_2-M_2)}{\varepsilon^2} (M_2-f_2) + O(\varepsilon)\\
&\le - \dfrac{(2\eta''({u})+O(\varepsilon))}{\varepsilon^2} (|f_1-{M}_1|^2+|f_2-{M}_2|^2)+O(\varepsilon)\\
&=- \dfrac{2\eta''({u})}{2 \lambda^2} |v-f(u)|^2+O(\varepsilon)\\
&=- \dfrac{\eta''({u})}{\lambda^2} |\varepsilon^2\partial_t v + \lambda^2 \partial_x u|^2+O(\varepsilon)\\
&=- \dfrac{\eta''(\bar{u})}{\lambda^2} |\varepsilon^2\partial_t v + \lambda^2 \partial_x u|^2+O(\varepsilon),
\end{aligned}
\end{equation}
where the last equalities follow from $\varepsilon^2\partial_t v + \lambda^2 \partial_x u=v-f(u)$ in (\ref{scaled_JinXin}) and  $O(\varepsilon^{-1} \|f_i-M_i\|^2_\infty)=O(\varepsilon)$, which is due to the $L^\infty$ bounds on $\textbf{f}^{\; \varepsilon}, \, \partial_t \textbf{f}^{\; \varepsilon}$.\\
Notice that equality $\eta''(u)=\eta''(\bar{u})$ holds since we chose a quadratic entropy $\eta(u)$ for system (\ref{hyperbolic_Jin_Xin}). For scalar conservation laws, there are indeed infinite entropy-entropy flux pairs $(\eta(\bar{u}), \, Q(\bar{u}))$ satisfying
$\eta'(\bar{u}) f'(\bar{u})=Q'(\bar{u}).$ We remind to \cite{Bianchini3} for an application of the method in a more constrained case.\\
Now we deal with $I_2$, by using the expansions of Lemma \ref{lemma-perturbed-maxwellians}.
\begin{equation}
\label{I2}
\begin{aligned}
I_2&=-\partial_t (\mathcal{H}_1(\overline{\mathcal{M}}_1)+\mathcal{H}_2(\overline{\mathcal{M}}_2))-\dfrac{\lambda}{\varepsilon} \partial_x (\mathcal{H}_1(\overline{\mathcal{M}}_1)-\mathcal{H}_2(\overline{\mathcal{M}}_2))\\
&=-[\mathcal{H}'_1(\overline{\mathcal{M}}_1) (\partial_t \overline{\mathcal{M}}_1 + \dfrac{\lambda}{\varepsilon} \partial_x \overline{\mathcal{M}}_1)-\mathcal{H}'_2(\overline{\mathcal{M}}_2) (\partial_t \overline{\mathcal{M}}_2 - \dfrac{\lambda}{\varepsilon} \partial_x \overline{\mathcal{M}}_2)]\\
&=-[\eta'(\bar{u})-\varepsilon \lambda \eta''(\bar{u}) \partial_x \bar{u}] (\partial_t \overline{\mathcal{M}}_1 + \dfrac{\lambda}{\varepsilon} \partial_x \overline{\mathcal{M}}_1)\\
&-[\eta'(\bar{u})+\varepsilon \lambda \eta''(\bar{u}) \partial_x \bar{u}] (\partial_t \overline{\mathcal{M}}_2 - \dfrac{\lambda}{\varepsilon} \partial_x \overline{\mathcal{M}}_2)+O(\varepsilon)\\
&=-\eta'(\bar{u})[\partial_t (\overline{\mathcal{M}}_1 + \overline{\mathcal{M}}_2) + \frac{\lambda}{\varepsilon} \partial_x (\overline{\mathcal{M}}_1-\overline{\mathcal{M}}_2)]\\
&- \varepsilon \lambda \eta''(\bar{u}) \partial_x \bar{u} [-\partial_t \overline{\mathcal{M}}_1+\partial_t \overline{\mathcal{M}}_2-\frac{\lambda}{\varepsilon} \partial_x \overline{\mathcal{M}}_1 - \frac{\lambda}{\varepsilon} \partial_x \overline{\mathcal{M}}_2]+O(\varepsilon)\\
&=-\eta'(\bar{u})[\partial_t \bar{u}+\partial_x f(\bar{u})-\lambda^2 \partial_{xx}\bar{u}]\\
&-\varepsilon \lambda\eta''(\bar{u}) \partial_x \bar{u}[-\frac{\lambda}{\varepsilon}\partial_x \bar{u}+O(\varepsilon)]+O(\varepsilon) \\
&= \eta''(\bar{u}) \lambda^2 (\partial_x \bar{u})^2+O(\varepsilon),
\end{aligned}
\end{equation}
where the last equalities are provided by the explicit expressions of $\overline{\mathcal{M}}_1, \overline{\mathcal{M}}_2$ in (\ref{limit_f}).
\begin{equation}
\label{I_3}
\begin{aligned}
I_3&=-\partial_t (\mathcal{H}_1'(\overline{\mathcal{M}}_1) (f_1-\overline{\mathcal{M}}_1))-\partial_t (\mathcal{H}_2'(\overline{\mathcal{M}}_2) (f_2-\overline{\mathcal{M}}_2))\\
&=-\partial_t[\eta'(\bar{u}) (u-\bar{u})-\varepsilon \lambda \eta''(\bar{u})\partial_x \bar{u}\cdot(f_1-f_2-(\overline{\mathcal{M}}_1-\overline{\mathcal{M}}_2))]+O(\varepsilon^2)\\
&=- \partial_t [\eta'(\bar{u}) (u-\bar{u})]+\varepsilon^2 \eta''(\bar{u}) \partial_x \bar{u} \cdot \partial_t[v-f(\bar{u})+\lambda^2 \partial_x \bar{u}] + O(\varepsilon)\\
&=-  \partial_t [\eta'(\bar{u}) (u-\bar{u})]+\eta''(\bar{u})\partial_x \bar{u} \cdot (\varepsilon^2 \partial_t v + \lambda^2 \partial_x u)-\eta''(\bar{u}) \lambda^2\partial_x \bar{u} \cdot \partial_x u+O(\varepsilon)\\
&=\eta''(\bar{u}) [\partial_x f(\bar{u})-\lambda^2 \partial_{xx}\bar{u}] \cdot (u-\bar{u})+\eta'(\bar{u}) \cdot (\partial_xv-\partial_xf(\bar{u})+\lambda^2 \partial_{xx} \bar{u})\\
&+\eta''(\bar{u}) \partial_x \bar{u} \cdot (\varepsilon^2 \partial_t v+\lambda^2 \partial_x u)-\eta''(\bar{u})\lambda^2 \partial_x \bar{u}\cdot  \partial_x u + O(\varepsilon).
\end{aligned}
\end{equation}
We are finally left with $I_4$.
$$
\begin{aligned}
I_4&=-\dfrac{\lambda}{\varepsilon}\partial_x(\mathcal{H}_1'(\overline{\mathcal{M}}_1)(f_1-\overline{\mathcal{M}}_1))+\dfrac{\lambda}{\varepsilon}\partial_x(\mathcal{H}_2'(\overline{\mathcal{M}}_1)(f_2-\overline{\mathcal{M}}_2))\\
&=\dfrac{\lambda}{\varepsilon} \partial_x [(-\eta'(\bar{u})+\varepsilon \lambda \eta''(\bar{u})\partial_x \bar{u}) \cdot (f_1-\overline{\mathcal{M}}_1)\\
&+(\eta'(\bar{u})+\varepsilon \lambda \eta''(\bar{u})\partial_x \bar{u}) \cdot (f_2-\overline{\mathcal{M}}_2)]+O(\varepsilon)\\
&=-\dfrac{\lambda}{\varepsilon} \partial_x[\eta'(\bar{u})\cdot (f_1-f_2-(\overline{\mathcal{M}}_1-\overline{\mathcal{M}}_2))]\\
&+\lambda^2 \eta''(\bar{u}) \partial_{x}[\partial_x \bar{u} \cdot (f_1+f_2-(\overline{\mathcal{M}}_1+\overline{\mathcal{M}}_2))]\\
\end{aligned}
$$
$$
\begin{aligned}
&=-\partial_x [\eta'(\bar{u}) \cdot (v-f(\bar{u})+\lambda^2 \partial_x \bar{u})]+\lambda^2\eta''(\bar{u}) \partial_x [\partial_x\bar{u} \cdot (u-\bar{u})]\\
&=-\eta'(\bar{u})\cdot (\partial_xv-\partial_xf(\bar{u})+\lambda^2 \partial_{xx} \bar{u}) - \eta''(\bar{u}) \partial_x \bar{u} \cdot (v-f(\bar{u})+\lambda^2 \partial_x \bar{u})\\
\end{aligned}
$$
\begin{equation}
\label{I_4}
\begin{aligned}
& + \lambda^2 \eta''(\bar{u}) \partial_{xx}\bar{u} \cdot (u-\bar{u}) + \lambda^2 \eta''(\bar{u}) \partial_{x}\bar{u} \cdot \partial_x(u-\bar{u})\\
&=-\eta'(\bar{u})\cdot (\partial_xv-\partial_xf(\bar{u})+\lambda^2 \partial_{xx} \bar{u}) - \eta''(\bar{u}) \partial_x \bar{u} \cdot (v-f(u))\\
& -\eta''(\bar{u}) \partial_x \bar{u} \cdot (f(u)-f(\bar{u}))-\lambda^2 \eta''(\bar{u}) (\partial_x \bar{u})^2\\
& + \lambda^2 \eta''(\bar{u}) \partial_{xx}\bar{u} \cdot (u-\bar{u}) + \lambda^2 \eta''(\bar{u}) \partial_x \bar{u} \cdot \partial_x u - \lambda^2 \eta''(\bar{u}) (\partial_x \bar{u})^2\\
&=-\eta'(\bar{u})\cdot (\partial_xv-\partial_xf(\bar{u})+\lambda^2 \partial_{xx} \bar{u}) +\eta''(\bar{u})\partial_x \bar{u} \cdot (\varepsilon^2 \partial_t v+\lambda^2 \partial_x u)\\
&-\eta''(\bar{u}) \partial_x \bar{u} \cdot (f(u)-f(\bar{u}))-2 \lambda^2 \eta''(\bar{u}) (\partial_x \bar{u})^2\\
&+ \lambda^2 \eta''(\bar{u}) \partial_{xx}\bar{u} \cdot (u-\bar{u}) + \lambda^2 \eta''(\bar{u}) \partial_x \bar{u} \cdot \partial_x u.
\end{aligned}
\end{equation}
As an intermediate step, let us look at the sum
\begin{equation}
\begin{aligned}
I_3+I_4&=2\eta''(\bar{u})\partial_x \bar{u} \cdot (\varepsilon^2 \partial_t v+\lambda^2 \partial_x u)\\
&-\eta''(\bar{u})\partial_x \bar{u} \cdot[f(u)-f(\bar{u})-f'(\bar{u}) \cdot (u-\bar{u})]\\
&-2 \lambda^2 \eta''(\bar{u}) (\partial_x \bar{u})^2.
\end{aligned}
\end{equation}
The total sum reads:
\begin{align*}
I_1&+I_2+I_3+I_4 \le -\dfrac{\eta''(\bar{u})}{\lambda^2}|\varepsilon^2\partial_t v+ \lambda^2 \partial_x u|^2+ \eta''(\bar{u}) \lambda^2 (\partial_x \bar{u})^2\\
&-2 \lambda^2 \eta''(\bar{u}) (\partial_x \bar{u})^2\\
&+2\eta''(\bar{u})\partial_x \bar{u} \cdot (\varepsilon^2 \partial_t v+\lambda^2 \partial_x u)\\
&-\eta''(\bar{u})\partial_x \bar{u} \cdot[f(u)-f(\bar{u})-f'(\bar{u}) \cdot (u-\bar{u})]\\
&=-\dfrac{\eta''(\bar{u})}{\lambda^2}|\varepsilon^2\partial_t v+ \lambda^2 \partial_x u|^2- \lambda^2 \eta''(\bar{u}) (\partial_x \bar{u})^2\\
&+2\eta''(\bar{u})\partial_x \bar{u} \cdot (\varepsilon^2 \partial_t v+\lambda^2 \partial_x u)\\
&-\eta''(\bar{u})\partial_x \bar{u} \cdot[f(u)-f(\bar{u})-f'(\bar{u}) \cdot (u-\bar{u})]\\
\end{align*}
\begin{equation}
\begin{aligned}
&=-\dfrac{\eta''(\bar{u})}{\lambda^2}|\varepsilon^2\partial_t v+ \lambda^2 \partial_x u|^2-\dfrac{\eta''(\bar{u})}{\lambda^2} |\lambda^2\partial_x \bar{u}+\varepsilon^2 \partial_t \bar{v}|^2+O(\varepsilon) \\
&+\dfrac{2}{\lambda^2}\eta''(\bar{u})\lambda^2 \partial_x \bar{u} \cdot (\varepsilon^2 \partial_t v+\lambda^2 \partial_x u)+\frac{2}{\lambda^2} \eta''(\bar{u}) \varepsilon^2 \partial_t \bar{v} \cdot (\varepsilon^2 \partial_t v + \lambda^2 \partial_x u)+O(\varepsilon)\\
&-\eta''(\bar{u})\partial_x \bar{u} \cdot[f(u)-f(\bar{u})-f'(\bar{u}) \cdot (u-\bar{u})]\\
&\le c(|\bar{u}|, \, |\partial_x \bar{u}|)|u-\bar{u}|^2 \\
&-\dfrac{\eta''(\bar{u})}{\lambda^2}(\varepsilon^2 \partial_t v + \lambda^2 \partial_x u - (\varepsilon^2 \partial_t \bar{v}+\lambda^2 \partial_x \bar{u}))^2+O(\varepsilon).
\end{aligned}
\end{equation}
We end up with the following estimate
\begin{align*}
\int_\mathbb{R}\tilde{\mathcal{H}}(\textbf{f}^{\; \varepsilon}| \overline{\textbf{f}})(t) \, dx & + \frac{1}{\lambda^2} \int_0^t \int_\mathbb{R} \eta''(\bar{u}) (\varepsilon^2 \partial_\tau v + \lambda^2 \partial_x u - (\varepsilon^2 \partial_\tau \bar{v}+\lambda^2 \partial_x \bar{u}))^2 \,  d\tau \, dx\\
& \le \int_0^t \int_\mathbb{R} c(|\bar{u}|, \, |\partial_x \bar{u}|)|u-\bar{u}|^2 \, dx \, ds + \int_\mathbb{R}\tilde{\mathcal{H}}(\textbf{f}^{\; \varepsilon}_0| \overline{\textbf{f}}_0) \, dx + O(\varepsilon)\\
&= \int_0^t\int_\mathbb{R} c(|\bar{u}|, \, |\partial_x \bar{u}|)|u-\bar{u}|^2 \, dx\, ds +O(\varepsilon),
\end{align*}
where $\tilde{\mathcal{H}}(\textbf{f}^{\; \varepsilon}_0| \overline{\textbf{f}}_0)=0$ thanks to the well-prepared initial data.\\
At this point, the definition of the relative entropy (\ref{relative_entropy_def}) and the expansion of Lemma \ref{lemma-perturbed-maxwellians} imply that
\begin{align*}
\tilde{\mathcal{H}}(\textbf{f}^{\; \varepsilon}| \bar{\textbf{f}})=\eta''(\bar{u})|u-\bar{u}|^2 + O(\varepsilon) \quad \text{in  } L^\infty_x,
\end{align*}
which ends the proof after an application of the Gronwall inequality.
\end{proof}
We are now ready to prove our result.
\begin{proof}[of Theorem \ref{thm-main}]
As a direct consequence of Lemma \ref{lemma_2_derivative_entropy}, the assumptions of Theorem \ref{thm:relative-entropy-estimate} are satisfied for $t \in [0, T^*]$, with $T^*=c(\|u_0\|_{s+1})$ defined in (\ref{def-time}). Therefore, Theorem \ref{thm:relative-entropy-estimate} provides the following estimate:
\begin{equation}\label{estimate-L2-relative-entropy}
\sup_{t \in [0, T^*]} \|(u-\bar{u})(t)\|_0 \le c(M) \sqrt{\varepsilon},
\end{equation}
where $T^*$ is defined in (\ref{def-time}) and $c(M)$ depends on $M$ in (\ref{def-time}). From (\ref{estimate-M}), we also recall that
$$\|u(t)\|_s \le  (M_0+O(\varepsilon)) e^{c(M) t} \le c_S M, \quad t\in [0, T^*],$$
i.e.
\begin{equation}\label{estimate-infty-1}
|u(t)|_\infty \le c_S (M_0+O(\varepsilon)) e^{c(M) t} \le c_S M, \quad t\in [0, T^*].
\end{equation}
The Interpolation Theorem for Sobolev spaces provides the following bound for any $s' \in (0, s)$:
\begin{align*}
\|(u-\bar{u})(t)\|_{s'} & \le \|(u-\bar{u})(t)\|_0^{1-s'/s} \|(u-\bar{u})(t)\|_s^{s'/s}\\
& \le c(M) \varepsilon^{\frac{s-s'}{2s}} (M_0+(M_0+O(\varepsilon)) e^{C(M)t})^{s'/s},
\end{align*}
where the last inequality follows from (\ref{estimate-L2-relative-entropy}) and (\ref{estimate-infty-1}). By using embedding properties,
\begin{align*}
|(u-\bar{u})(t)|_{\infty} \le c_S \|(u-\bar{u})(t)\|_{s'} \le c(M) \varepsilon^{\frac{s-s'}{2s}} (M_0+(M_0+O(\varepsilon)) e^{C(M)t})^{s'/s},
\end{align*}
i.e.
\begin{equation}\label{estimate-infty-11}
\begin{aligned}
|u(t)|_\infty & \le |\bar{u}|_\infty + c(M) \varepsilon^{\frac{s-s'}{2s}} (M_0+(M_0+O(\varepsilon)) {e^{C(M)t})}^{s'/s}\\
& \le |u_0|_\infty + c(M) \varepsilon^{\frac{s-s'}{2s}} (M_0+(M_0+O(\varepsilon)) {e^{C(M)t})}^{s'/s}\\
& \le c_S M_0 + c(M) \varepsilon^{\frac{s-s'}{2s}} (M_0+(M_0+O(\varepsilon)) {e^{C(M)t})}^{s'/s}.
\end{aligned}
\end{equation}
Recalling now the definition of $T^*$ in (\ref{def-time}) and choosing $M=4M_0$, estimate (\ref{estimate-infty-11}) implies that, for $t \le T^*$,
\begin{equation}\label{ineq:Linfty}
|u(t)|_\infty \le c_S M_0 + c\varepsilon^{\frac{1}{2}-\delta} < 2 c_S M_0=c_S \frac{M}{2}, \quad \delta=\frac{s'}{2s}.
\end{equation}
Now, let us assume that $T^* < T^\varepsilon$. Therefore, by definition of $T^*$ in (\ref{def-time}), one gets
\begin{align*}
|u(T^*)|_\infty = 4 c_S M_0 = c_S M.
\end{align*}
On the other hand, from (\ref{ineq:Linfty}), there exists an arbitrarily small fixed $\varepsilon_0$ such that, for all $\varepsilon \le \varepsilon_0$,
\begin{align*}
|u(T^*)|_\infty \le c_S M_0 + c \varepsilon^{\frac{1}{2}-\delta} \le c_S \frac{M}{2}.
\end{align*}
Therefore $T^* \ge T_\varepsilon$ by contradiction and $\sup_{t \in [0, T]} |u(t)|_\infty$ is bounded for any fixed interval of time $[0,T]$ by employing a classical continuation argument. The same reasoning applies to the relative entropy estimate in Theorem \ref{thm:relative-entropy-estimate}, which holds as far as the derivatives of the Maxwellians are positive (as far as $u$ is bounded) as in Lemma \ref{lemma-positive-Max}, and then the strong convergence result is valid for any fixed interval of time $[0, T]$.
\end{proof}
\section*{Conclusive remark}
We point out that the main difficulty in applying this method to general diffusive vector-BGK models as the ones introduced in \cite{BGN} is to find a symmetrizer as $\Sigma$ in Section \ref{sec:nonlinear} to get uniform bounds of the densities $f_i^\varepsilon$, as done in \cite{Bianchini1} for 2D Navier-Stokes. The rest, up to heavy computations, is straightforward.

\section*{Acknowledgements}
The author thanks Roberto Natalini for useful comments on the Introduction of this work.\\
This paper was partially funded by the GNAMPA (INdAM) project \emph{Partially dissipative hyperbolic systems
with applications to biological models} 2019.\\
This project has received funding from the European Research Council (ERC) under the European Union's Horizon 2020 research and innovation program Grant agreement No 637653, project BLOC ``Mathematical Study of Boundary Layers in Oceanic Motion’’. This work was  supported by the SingFlows project, grant ANR-18-CE40-0027 of the French National Research Agency (ANR).

%
%
%

\end{document}